\documentclass[11pt]{article}

\usepackage{amsmath, amsthm, amssymb, enumitem}

\newcommand{\ST}{Szemer\'{e}di-Trotter} 

\newcommand{\F}{\ensuremath{\mathbb{F}}}

\newcommand{\R}{\ensuremath{\mathbb{R}}}
\newcommand{\Z}{\ensuremath{\mathbb{Z}}}

\theoremstyle{plain}
\newtheorem{thm}{Theorem}
\newtheorem{lem}[thm]{Lemma}

\newtheorem{cor}[thm]{Corollary}

\theoremstyle{definition}

\newtheorem*{definition*}{Definition}

\newtheorem*{exercise}{Exercise}
\newtheorem*{acknowledgements}{Acknowledgements}
\theoremstyle{remark}

\newtheorem*{note}{Note}

\newtheorem*{question*}{Question}

\title{A Second Wave of Expanders in Finite Fields}
\author{Brendan Murphy and Giorgis Petridis}
\date{\today}

\begin{document}

\maketitle

\abstract{
This is an expository survey on recent sum-product results in finite fields.

We present a number of sum-product or ``expander'' results that say that if $|A| > p^{2/3}$ then some set determined by sums and product of elements of $A$ is nearly as large as possible, and if $|A|<p^{2/3}$ then the set in question is significantly larger that $A$.
These results are based on a point-plane incidence bound of Rudnev, and are quantitatively stronger than a wave of earlier results following Bourgain, Katz, and Tao's breakthrough sum-product result.

In addition, we present two geometric results: an incidence bound due to Stevens and de Zeeuw, and bound on collinear triples, and an example of an expander that breaks the threshold of $p^{2/3}$ required by the other results.

We have simplified proofs wherever possible, and hope that this survey may serve as a compact guide to recent advances in arithmetic combinatorics over finite fields.
We do not claim originality for any of the results.
}


\section{Introduction}
\label{sec:orgheadline2}
This is an expository survey of recent results related to the \emph{sum-product problem over finite fields}.
Roughly speaking, the sum-product problem is to show that a finite subset of a field cannot have both additive and multiplicative structure (unless it is essentially a subfield).
For instance, if \(p\) is prime and \(A\) is a subset of the field \(\F_p\) with \(p\) elements, then we would expect the set
\[
A+AA := \{a_1+a_2a_3\colon a_1,a_2,a_3\in A\}
\]
to be much larger than \(|A|\), since \(\F_p\) has no non-trivial subfields.

In general, we will consider polynomials \(f\in \Z[x_1,\ldots,x_n]\) and ask if there is a $\delta >0$ such that
\[ |f(A,\ldots,A)| \geq |A|^{1+\delta}\]
for all ``small'' subsets \(A\) of \(\F_p\).
We will call such polynomials \emph{expanding polynomials} or \emph{expanders}.

Explicit examples of expanding polynomials were first given in characteristic zero \cite{erdos1983products-a,elekes1997number}.
The arguments employed here typically use topological properties of the underlying field---for instance, the order of the integers or reals.
Over finite fields, such as $\F_p=\Z/p\Z$, such properties are unavailable, and expansion results are more difficult to prove.
Using Fourier analysis in $\F_p$, Garaev~\cite{garaev2008sum-product} showed that for $A\subseteq\F_p$
\begin{equation}
  \label{eq:29}
  \max(|A+A|,|AA|)\gg \min \left( \sqrt{p|A|},\frac{|A|^2}{p^{1/2}} \right),
\end{equation}
which is optimal for $|A|>p^{2/3}$ and trivial for $|A|<p^{1/2}$.

Bourgain, Katz, and Tao~\cite{bourgain2004sum-product-a} proved the first non-trivial sum-product estimate for ``small'' subsets of finite fields.
They showed that if $A$ is a subset of the prime field $\F_p$ such that $p^\alpha < |A| < p^{1-\alpha}$ for some $\alpha >0$, then there is some $\epsilon > 0$ depending on $\alpha$ such that
\begin{equation}
  \label{eq:20}
  \max(|A+A|,|AA|)\gg |A|^{1+\epsilon}.
\end{equation}
The bounds on $|A|$ rule out the possibility that $|A\cap \F|\gg p^{-\alpha}\max(|A|,|F|)$ for any subfield $\F$ of $\F_p$ (i.e. for $\F=\{0\}, \F=\F_p$); in general, it is true that there is a non-trivial sum-product estimate for $A\subseteq\F_q$ as long as $A$ is not ``roughly equivalent'' to a subfield.
The estimate \eqref{eq:20} still holds when the lower bound on $|A|$ is dropped---this is due to Glibichuk and Konyagin \cite{glibichuk2007additive}.

Garaev \cite{garaev2007explicit} found the first explicit value of $\epsilon$, which was then improved by several authors \cite{katz2008slight,bourgain2009variant,rudnev2012improved}, finally resulting in the lower bound
\[
\max(|A+A|,|AA|)\gg |A|^{1+1/11}(\log |A|)^{-4/11}.
\]
The method behind these early sum-product results for finite fields is called the \emph{pivot method}.
The pivot method is essentially algebraic; it is a flexible method, but it is quantitatively inefficient.

Recently, a new \emph{geometric}\/ method for proving sum-product results in finite fields was discovered.
This geometric method is based on a point-plane incidence bound of Rudnev \cite{rudnev2014number}.
Rudnev's bound has ushered in a new wave of expander results.

For instance, Roche-Newton, Rudnev, and Shkredov  \cite{roche-newton2016sum-product} applied Rudnev's bound to show that if $A$ is a subset of $\F_p$ with $|A|<p^{2/3}$, then
\[
\max(|A+A|,|AA|)\gg |A|^{1+1/5}.
\]
Even more impressive is their lower bound for the mixed sum-product set $A+AA$: for $A\subseteq\F_q$
\begin{equation}
  \label{eq:27}
|A+AA|\gg \min(|A|^{3/2},p)
\end{equation}
where again $p$ is the characteristic of the field $\F_q$.
For $|A|<p^{2/3}$, this bound matches what can be proved directly by the \ST{} incidence bound over $\R$, namely
\begin{equation}
  \label{eq:25}
  |A+AA| \gg |A|^{3/2}
\end{equation}
for all finite subsets $A\subseteq\R$.
The bound \eqref{eq:25} that has only been slightly improved over $\R$ \cite{shkredov2016question}, thus Rudnev's point-plane incidence bound allows us to prove expander results that nearly match those known over the real numbers.

A number of similar results have followed from Rudnev's point-plane bound.
These results are often of the form $|f(A^k)|\gg \min(|A|^{3/2},p)$ for some polynomial $f\in \Z[x_1,\ldots,x_k]$; thus if $|A|>p^{2/3}$, then $|f(A^k)|\gg p$.
We say that these results are at the ``$p^{2/3}$ threshold'':
\begin{enumerate}
\item $|AA+AA|\gg \min(p,|A|^{3/2})$ (Rudnev \cite{rudnev2014number})
\item $|(A-A)(A-A)|\gg p$ if $|A|>p^{2/3}$ (Bennett, Hart, Iosevich, Pakianathan, and Rudnev \cite{bennett2013group}, see also \cite{hart2007sum-product})
\item $|(A-A)^2+(A-A)^2|\gg \min(p,|A|^{3/2})$ (Petridis \cite{petridis2016pinned}, see also \cite{chapman2009pinned})
\item $|A+AA|\gg \min(p, |A|^{3/2})$ (Roche-Newton, Rudnev, and Shkredov \cite{roche-newton2016sum-product})
\item $|A(A+A)|\gg\min(p, |A|^{3/2})$ (Aksoy-Yazici, Murphy, Rudnev, and Shkredov \cite{yazici2015growth})
\end{enumerate}
In the last section of the paper, we present an expander result below the $p^{2/3}$ threshold.
Namely, that if $|A|>p^{5/8}$, then
\begin{equation}
  \label{eq:35}
  |(A-A)(A-A)|\gg p^{5/8}.
\end{equation}
This result is due the second author \cite{petridis2016products}.
As an expander result, this says that the polynomial $f(x,y,z,w)=(x-y)(z-w)$ satisfies $|f(A^4)|\gg p$ whenever $|A|> p^{5/8}$.

In this survey, we take Rudnev's point-plane incidence bound as a black-box and use it to prove a variety of sum-product estimates.
We have tried to present the cleanest possible proofs, and have chosen results that illustrate the how to apply the point-plane incidence bound in a variety of situations.
We do not claim originality for any of the results.

In Section~\ref{sec:geom-appr-sum}, we introduce Rudnev's point plane incidence bound, and use it prove that $|A+AA|\gg \min(p,|A|^{3/2})$.
This method of proof will be a model for many later arguments.
The section ends with a generalization of the method, due to \cite{yazici2015growth}, phrased in terms of certain ``energies'' $E(Q;A)$ or $E(L,A)$, where $A\subseteq\F_p$, $Q\subseteq\F_p^2$, and $L$ is a collection of lines in $\F_p^2$.

This generalized argument will be applied in Section~\ref{sec:orgheadline18} to prove two further expander results, and in Section~\ref{sec:incid-results-points} to prove two geometric results: an incidence bound due to Stevens and de Zeeuw, and a bound on ``collinear triples'' due to Aksoy-Yazici, Murphy, Rudnev, and Shkredov \cite{yazici2015growth}.

The final section of paper contains a proof of the expansion result \eqref{eq:35}, which seems to be the first such result below the $p^{2/3}$ threshold.

\begin{acknowledgements}
  We thank Olly Roche-Newton, Misha Rudnev, and Sophie Stevens for several helpful suggestions.
We would also like to thank Mel Nathanson for inviting us to write this survey for the proceedings CANT 2015/2016.
\end{acknowledgements}

\section{A geometric approach to sum-product problems in finite fields}
\label{sec:geom-appr-sum}
In this section, we present a proof of \eqref{eq:27} based on Rudnev's point plane incidence bound, which will serve as a prototype for further applications.
We then generalize the method of proof; this generalized formulation will be applied to a variety of applications in the remaining sections.

\subsection{Rudnev's point-plane incidence bound}

Rudnev's incidence bound is the following.
\begin{thm}[Rudnev \cite{rudnev2014number}]
  \label{thm:1}
Let $\F$ denote a field, and let $p$ denote the characteristic of $\F$.
Let $P$ be a set of points in $\F^3$ and let $\Pi$ be a set of planes in $\F^3$ with $|P|\leq |\Pi|$.
If $p > 0$, assume that $|P|\ll p^{2}$.
Let $k$ denote the maximum number of points of $P$ contained in a line.
Then
\[
I(P,\Pi)\ll |P|^{1/2}|\Pi|+k|P|.
\]
\end{thm}
Theorem~\ref{thm:1} is strongest when $|P|=|\Pi|$.
See \cite{zeeuw2016short} for a short proof of Theorem~\ref{thm:1}, due to de Zeeuw.

For convenience, we combine Theorem~\ref{thm:1} with an incidence bound for large subsets of $\F_P^2$.
\begin{cor}
\label{cor:1}
Let \(p\) be an odd prime, let \(P\) be a collection of points in \(\F_p^3\), and let \(\Pi\) be a collection of planes in \(\F_p^3\).

Suppose that \(|P|=|\Pi|=N\) and that at most \(k\) points of \(P\) are collinear.
Then the number of point-plane incidences satisfies
\[
I(P,\Pi) \ll \frac{N^2}p + N^{3/2} + kN.
\]
\end{cor}
The advantage of Corollary~\ref{cor:1} over Theorem~\ref{thm:1} is that we do not need to bound the size of the point set and the collection of planes before applying the bound.
\begin{proof}
  By \cite{murphy2016point-line} (see also \cite{hart2011averages,vinh2011szemeredi-trotter,lund2014incidence}), we have
\[
I(P,\Pi) \leq \frac{|P||\Pi|}p + p\sqrt {|P||\Pi|} = \frac{N^2}p + pN.
\]
Thus if $N > p^2$, then
\[
I(P,\Pi) \ll \frac{N^2}p.
\]
On the other hand, if $N<p^2$, then by Theorem~\ref{thm:1} we have
\[
I(P,\Pi) \ll  |P|^{1/2}|\Pi|+k|P| = N^{3/2}+kN.
\]
\end{proof}

\subsection{A lower bound for \(|A+AA|\)}
\label{sec:orgheadline7}

In this section, we prove the following theorem, due to Roche-Newton, Rudnev, and Shkredov \cite{roche-newton2016sum-product}.
\begin{thm}
\label{thm:2}
  For all subsets $A$ of $\F_p$, we have
\[
|A+AA| \gg \min(p, |A|^{3/2}).
\]
\end{thm}
The proof of Theorem~\ref{thm:2} will serve as a model for the rest of the results in this section.

\begin{proof}
First, we apply Cauchy-Schwarz.
Let
\[
 r_{A+AA}(x)=|\{(a,b,c)\in A^3\colon a+bc =x\}|.
 \]
The support of \(r_{A+AA}\) is \(|A+AA|\) and
\[
 \sum_x r_{A+AA}(x) = |A|^3,
 \]
thus by Cauchy-Schwarz
\[ |A|^6 = \left( \sum_x r_{A+AA}(x) \right)^2 \leq |A+AA|\sum_{x}r^2_{A+AA}(x).\]
To show that
\[ |A+AA| \gg \min(p, |A|^{3/2})\]
it suffices to show that
\[
 \sum_{x}r^2_{A+AA}(x)\ll \max\left( \frac{|A|^6}p, |A|^{9/2}\right).
 \]

Next, we reduce the problem to a point-plane incidence problem.
The second moment of \(r_{A+AA}(x)\) counts the number of solutions to the equation
\begin{equation}
  \label{eq:1}
  a+bc=a'+b'c'
\end{equation}
with \(a,b,c,a',b',c'\) in \(A\).

To bound the number of solutions to this equation, we will realize the each solution as an incidence between a certain point and a certain plane.
Let \(\pi_{a,b,c'}\) denote the set of points \((x,y,z)\) such that
\[
a = x -by + c'z.
\]
The point \((x,y,z)=(a',c,b')\) is incident to the plane \(\pi_{a,b,c'}\) precisely when~\eqref{eq:1} is satisfied: if
\[
a = a' - bc + c'b',
\]
then
\[
a+bc = a'+b'c'.
\]

Finally, we apply Rudnev's point-plane incidence bound, in the form of Corollary~\ref{cor:1}.
Let \(P=\{(a',c,b')\in A^3\}\) and let \(\Pi=\{\pi_{a,b,c'}\colon (a,b,c')\in A^3\).
Then \(|P|=|\Pi|=|A|^3\).
Thus by Corollary~\ref{cor:1}, we have
\[
I(P,\Pi)\ll \frac{|A|^6}p + |A|^{9/2} + k |A|^3.
\]
This yields the desired upper bound on the second moment of \(r_{A+AA}(x)\), provided that the number \(k\) of collinear points of \(P=A\times A\times A\) is not too large.

It is not hard to show that \(k\leq |A|\): if \(\ell\) is parallel to the $x$-axis, then \(|P\cap \ell|\leq |A|\), while if \(\ell\) is not parallel to the $x$-axis, then \(\ell\) may be parameterized in terms of \(y\) or $z$, which again implies that \(|P\cap\ell|\leq |A|\).

Since \(k|A|^3 \leq |A|^4 \leq |A|^{9/2}\), we have
\begin{align*}
  \sum_x r_{A+AA}^2(x) = I(P,\Pi) &\ll \frac{|A|^6}p + |A|^{9/2} + k |A|^3\\
& \ll \max\left(\frac{|A|^6}p, |A|^{9/2}\right),
\end{align*}
as desired.
\end{proof}

\subsection{Generalizing the method}
\label{sec:orgheadline9}

In this section, we will generalize the method used to count solutions to \eqref{eq:1}.
This generalization first appeared in \cite{yazici2015growth}; below we present simplification of the original argument. 

In order to form the set of points and planes associated to the equation~\eqref{eq:1}
\[
a+bc=a'+b'c' 
\]
it was essential that \((a,c)\) was independent from \(b\) and \((a',c')\) was independent from \(b'\).
While we also knew that \(a\) and \(c\) were independent, we do not make use of this in forming the points and planes.

Given a set of pairs $Q\subseteq\F_p^2$ and a set $A\subseteq\F_p$, let $E(Q;A)$ denote the number of solutions to
\begin{equation}
  \label{eq:2}
  ma+b = m'a'+b'
\end{equation}
with \((m,b),(m',b')\) in \(Q\) and \(a,a'\) in \(A\).
\begin{thm}
  \label{thm:3}
\[
E(Q;A) \ll \frac{|Q|^2|A|^2}p + (|Q||A|)^{3/2} + k|Q||A|,
\]
where
\[
k\leq \max\left(|A|, \max_{\ell\text{ line in }\F^2}|Q\cap \ell|\right).
\]
\end{thm}

\begin{proof}
For each $(m,b)$ in $Q$ and $a$ in $A$, form a plane
\[
\pi_{(m,b),a'} = \{(x,y,z)\in \F_q^3\colon mx+b = ya'+z\}.
\]
The equation~\eqref{eq:2} holds if and only if \((a,m',b')\in\pi_{(m,b),a'}\).

If we let \(P= A\times Q\) and let \(\Pi\) denote the set of all planes \(\pi_{(m,b),a'}\) with \((m,b)\) in \(Q\) and \(a'\) in \(A\).
Then \(|P|=|\Pi|\), so we have
\[
I(P,\Pi) \ll  \frac{|P|^2}p + |P|^{3/2} + k |P|.
\]

To bound $k$, we argue as before: if the $x$-coordinate of $\ell$ is not constant, then $|P\cap\ell|\leq |A|$, since we may parameterize $\ell$ in terms of $x$, and $P=A\times Q$.
If the $x$-coordinate of $\ell$ is constant (say equal to $a_0$), then
\[
|P\cap\ell| \leq |\{a_0\}\times Q\cap\ell|\leq \max_{\ell\text{ line in }\F^2}|Q\cap\ell|.
\]
\end{proof}

\subsection{A bound for the energy of affine transformations acting on the line}
\label{sec:orgheadline8}
In \cite{yazici2015growth}, the points in \(Q\) were associated to lines by duality.
There is a natural interpretation of this dual quantity, however the proof is more convoluted.
Now that we have the bound for \eqref{eq:2} in hand, we can give the dual version quite easily.

To each point \((m,b)\) in \(Q\), we associate an affine transformation \(\ell_{m,b}\) defined by \(\ell_{m,b}(x)=mx+b\).
We let \(L_Q\) denote the set of all \(\ell_{m,b}\) with \((m,b)\) in \(Q\).
With this notation, equation \eqref{eq:2} counts the number of solutions to
\begin{equation}
  \label{eq:3}
  \ell(a)=\ell'(a')
\end{equation}
with \(\ell,\ell'\) in \(L_Q\) and \(a,a'\) in \(A\).
We use \(E(L,A)\) to denote the number of solutions to \eqref{eq:3}.

\begin{cor}
\label{cor:2}
Let $L$ be a set of lines in $\F_p^2$ and let $A$ be a subset of  $\F_p$.
Let $\kappa$ denote the size of the largest pencil of lines in $L$; that is, $\kappa$ is maximum size of a subset $L'\subseteq L$ such that all of the lines of $L'$ are parallel or pass through a common point.

Then 
\[
E(L,A)\ll \frac{|L|^2|A|^2}p + (|L||A|)^{3/2} + k|L||A|,
\]
where $k\leq \max(|A|,\kappa)$.
\end{cor}
\begin{proof}
Let $Q$ be such that $L=L_Q$.
Then
\[
E(L,A) = E(Q;A)
\]
and $k$ is the maximum of $|A|$ and the maximum number of points of $Q$ lying on a line, which is precisely maximum number of lines in a pencil.
\end{proof}

The quantity $E(L,A)$, which is the number of solutions to
\[
\ell(a)=\ell'(a')\qquad \ell,\ell'\in L, a,a'\in A
\]
is analogous to the \emph{multiplicative energy} $E^\times(B,A)$ of a set \(B\) and a set \(A\), which is the number of solutions to
\[
ba=b'a'\qquad b,b'\in B, a,a'\in A.
\]

\section{Expansion results at the \(p^{2/3}\) threshold}
\label{sec:orgheadline18}

\subsection{A lower bound for \(|A(A+A)|\)}

\begin{thm}
  \label{thm:4}
For any subset $A$ of $\F_p$, we have
\[
|A(A+A)|\gg \min(p, |A|^{3/2}).
\]
\end{thm}

\begin{proof}
Without loss of generality, suppose that $A$ does not contain $0$.

  By Cauchy-Schwarz, we have
  \begin{equation}
    \label{eq:8}
    |A|^6\leq |A(A+A)|\,|\{(a,b,c,a',b',c')\in A^6\colon a(b+c)=a'(b'+c')\}|.
  \end{equation}
We wish to bound the number of solutions to
\begin{equation}
  \label{eq:7}
  a(b+c)=a'(b'+c')
\end{equation}
with $a,\ldots,c'$ in $A$.

Since we can write $a(b+c) = ab+ac$, if we let $Q=\{(a,ac)\colon a,c\in A\}$, then the number of solutions to \eqref{eq:7} is $E(Q;A)$.
The map $(a,c)\mapsto (a,ac)$ is injective, as long as $a\not=0$, so $|Q|=|A|^2$.
At most $|A|$ elements of $Q$ lie on a single line, so by Theorem~\ref{thm:3}, the number of solutions to \eqref{eq:7} is 
\[
|\{(a,b,c,a',b',c')\in A^6\colon a(b+c)=a'(b'+c')\}|\ll \frac{|A|^6}p + |A|^{9/2}.
\]
Combining this bound with \eqref{eq:8} yields the desired lower bound on $|A(A+A)|$.
\end{proof}

\begin{note}
The set of points $Q=\{(a,ac)\colon a,c\in A\}$ is projectively equivalent to $A\times A$, which immediately implies that $|Q\cap \ell|\leq |A|$ for any line $\ell$.
In general, if $Q$ is projectively equivalent to $B\times C$, then we have $k\leq \max(|A|,|B|,|C|)$.
\end{note}

The following example, suggested by Roche-Newton, can be proved by a similar argument.
\begin{exercise}
Let
\[
A(AA+1)=\{a(bc+1)\colon a,b,c\in A\}.
\]
Show that
\[
|A(AA+1)|\gg\min(p,|A|^{3/2}).
\]
\end{exercise}

\subsection{A lower bound for \(|(A-A)^2+(A-A)^2|\)}

In this section, we show that there is a point $(u,v)$ in $A\times A$ such that
\begin{equation}
  \label{eq:6}
  |(A-u)^2 + (A-v)^2|\gg \min(p,|A|^{3/2}).
\end{equation}
This result is due to the second author~\cite{petridis2016pinned}.

Geometrically, equation~\eqref{eq:6} says that the product set $P=A\times A$ determines $\gg\min(p, |P|^{3/4})$ distances to the point $(u,v)\in P$.

\begin{proof}
To prove a lower bound for $|(A-u)^2+(A-v)^2|$, we will bound the number of solutions to
  \begin{equation}
\label{eq:4}
    (a-u)^2+(b-v)^2 =(c-u)^2+(d-v)^2\qquad a,b,c,d,u,v\in A.
  \end{equation}
Then we will pigeonhole over $u$ and $v$, and apply a Cauchy-Schwarz energy type argument.

To bound the number of solutions to \eqref{eq:4}, we rearrange the equation
\[
  (a-u)^2-(c-u)^2 =(d-v)^2-(b-v)^2
\]
and simplify
\begin{equation}
\label{eq:5}
  a^2-c^2 -2(a-c)u = d^2-b^2 - 2(d-b)v.
\end{equation}
Equation~\eqref{eq:5} is linear in $u$ and $u$ is independent from $a,c$, similarly for $v,b,d$, so we might hope to apply Theorem~\ref{thm:3}.

Let 
\[
Q=\{(-2(a-c),a^2-c^2)\colon a,c\in A\}.
\]
Then the number of solutions to \eqref{eq:5} is $E(Q;A)$.

Note that $|Q|=|A|^2$, since the map
\[
(a,c)\mapsto (-2(a-c), a^2-c^2)
\]
is invertible.

Further, at most $2|A|$ points of $Q$ are contained in a single line, since for fixed $\alpha,\beta,\gamma$ the number of solutions to
\[
\alpha[-2(a-c)]+\beta(a^2-c^2)=\gamma
\]
is bounded by the maximum number of pairs $(a,c)$ of $A\times A$ that are contained in the quadratic curve
\[
-\alpha(x-y)+\beta(x^2-y^2)=\gamma.
\]
Given any $x$, there are at most two solutions for $y$.

Thus by Theorem~\ref{thm:3}, the number of solutions to \eqref{eq:4} is at most
\[
E(Q;A)\ll \frac{|A|^6}p+ (|A|^3)^{3/2} + 2|A|^3 \ll \frac{|A|^6}p+ |A|^{9/2}.
\]

By the pigeonhole principle, it follows that there is a pair $(u,v)$ in $A\times A$ such that the number of solutions to
\[
    (a-u)^2+(b-v)^2 =(c-u)^2+(d-v)^2\qquad a,b,c,d\in A
\]
is at most $O(|A|^4/p + |A|^{7/2})$.

By Cauchy-Schwarz we have
\[
|A|^4 \ll |(A-u)^2+(A-v)^2|\cdot \max(|A|^4/p,|A|^{7/2}),
\]
which implies the desired lower bound.
\end{proof}

See \cite{pham2016three-variable} for a generalization of this result to higher dimensions, as well as a general result on expanding quadratic polynomials.

\section{Incidence results for points and lines in $\F_p^2$}
\label{sec:incid-results-points}

\subsection{An incidence bound for Cartesian product point sets \(P=A\times B\)}

The following incidence bound is due to Stevens and de Zeeuw~\cite{stevens2016improved}.
\begin{thm}
  \label{thm:5}
Let $A$ and $B$ be subsets of $\F_p$ with $|A|\leq |B|$.
If $P=A\times B$ and $L$ is a set of lines in $\F_p^2$, then
\[
I(P,L) \ll \frac{|A||B|^{1/2}|L|}{p^{1/2}} + |A|^{3/4}|B|^{1/2}|L|^{3/4} + |P|^{2/3}|L|^{2/3} + |L|.
\]
\end{thm}
In particular, if $P=A\times A$, we have
\begin{equation}
  \label{eq:23}
  I(P,L)\ll \frac{|P|^{3/4}|L|}{p^{1/2}} + |P|^{5/8}|L|^{3/4} + |P| + |L|,
\end{equation}
since $|P|^{2/3}|L|^{2/3} > |P|^{5/8}|L|^{3/4}$ only when $|L| < |P|^{1/2}$, but in this case we have $I(P,L)\ll |P|$.
Further, if $|A||L| \ll p^2$, then the first term of \eqref{eq:23} is smaller than the second, so we have
\begin{equation}
  \label{eq:24}
    I(P,L)\ll |P|^{5/8}|L|^{3/4} + |P| + |L|.
\end{equation}

Before we prove Theorem~\ref{thm:5}, we prove a lemma that gives the correct leading terms.
\begin{lem}
\label{lem:1}
  For $P=A\times B$, as above, and any set of lines $L$, we have
\[
I(P,L) \leq |B|^{1/2}E(L,A)^{1/2}.
\]
\end{lem}
Thus
\[
I(P,L)\ll \frac{|A||B|^{1/2}|L|}{p^{1/2}} + |A|^{3/4}|B|^{1/2}|L|^{3/4} + k(|A||B||L|)^{1/2}.
\]
A priori, we have no control over $k$, so Theorem~\ref{thm:5} does not follow immediately from Lemma~\ref{lem:1}.
\begin{proof}[Proof of Lemma~\ref{lem:1}]
  We have
\[
I(P,L) = |\{(a,b,\ell)\in A\times B\times L\colon b=\ell(a)\}| = \sum_{b\in B}|\{(a,\ell)\in A\times L\colon b=\ell(a)\}|.
\]
Thus by Cauchy-Schwarz,
\[
I(P,L)\leq |B|^{1/2} \left(\sum_b|\{(a,\ell)\in A\times L\colon b=\ell(a)\}|^2\right)^{1/2}.
\]
The sum over all $b$ in $\F_p$ is equal to $E(L,A)$; that is, it is equal to the number of solutions to
\[
\ell(a)=\ell'(a')
\]
with $\ell,\ell'$ in $L$ and $a,a'$ in $A$.
Thus
\[
I(P,L)\leq |B|^{1/2}E(L,A)^{1/2}.
\]
\end{proof}
To apply Lemma~\ref{lem:1}, we need to make sure that not too many lines of $L$ lie in a pencil.
\begin{proof}[Proof of Theorem~\ref{thm:5}]
Let $k > 0$ be a parameter that we will choose later.

We begin by pruning large pencils of lines from $L$.
Suppose that $L$ contains a pencil $P_1$ with more than $k$ lines.
This pencil contributes at most $|A||B| + |P_1|$ incidences.
Let $L_1=L\setminus P_1$.
We continue pruning pencils until we reach a set of lines $L'$ that contains no pencils of size greater than $k$.
This process takes at most $|L|/k$ steps, hence the lines removed contribute at most
\[
\sum_{i=1}^{|L|/k}(|A||B|+|P_i|) = \frac{|A||B||L|}k + |L|
\]
incidences.

By Lemma~\ref{lem:1} and Corollary~\ref{cor:2}, we have
\begin{align*}
  I(P,L') \leq |B|^{1/2}E(L,A)^{1/2}
 &\ll |B|^{1/2} \left( \frac{|L|^2|A|^2}p + (|L||A|)^{3/2} + k|L||A| \right)^{1/2}\\
&\ll\frac{|A||B|^{1/2}|L|}{p^{1/2}} + |A|^{3/4}|B|^{1/2}|L|^{3/4} + \sqrt{k|A||B||L|}.
\end{align*}

Since $I(P,L) = I(P,L') + I(P, L\setminus L')$, we have
\[
I(P,L)\ll \frac{|A||B|^{1/2}|L|}{p^{1/2}} + |A|^{3/4}|B|^{1/2}|L|^{3/4}+ \sqrt{k|A||B||L|} + \frac{|A||B||L|}k+|L|.
\]
Setting $k=(|A||B||L|)^{1/3}$ yields
\[
I(P,L) \ll \frac{|A||B|^{1/2}|L|}{p^{1/2}} + |A|^{3/4}|B|^{1/2}|L|^{3/4} + |P|^{2/3}|L|^{2/3} + |L|.
\]
\end{proof}

Note that we have
\[
I(A\times B, L)\leq |A||L| + |A||B|,
\]
so we have
\begin{enumerate}
\item $I(A\times B, L)\ll |A||B|^{1/2}|L|/p^{1/2}$ if $|A||L| > p^2$,
\item $I(A\times B, L)\ll |A|^{3/4}|B|^{1/2}|L|^{3/4}$ if $|B|^2 < |A||L| < p^2$, and
\item $I(A\times B, L)\ll |A||L| + |P|$ if $|A||L| < |B|^2$.
\end{enumerate}

\begin{exercise}
Theorem~\ref{thm:5} can be used to prove a number of sum-product results using Elekes' method \cite{elekes1997number}.
\begin{enumerate}
\item Use the lines $\ell_{a,b}(t)=a(t+b)$ with $a,b\in A$ and the point set $P=A\times A(A+A)$ to show that \[|A(A+A)|\gg \min(p, |A|^{3/2}).\]
\item Use the lines $\ell_{a,b}(t)=at+b$ with $a,b\in A$ and the point set $P=A\times (A+AA)$ to show that \[|A+AA|\gg \min(p, |A|^{3/2}).\]
\item Use the lines $\ell_{a,b}(t)=t/a + b$ or $\ell_{a,b}(t)=a(t-b)$ and a point set of the form $P=AA\times (A+A)$ or $P=(A+A)\times AA$ to show that 
\[
\max(|A+A|,|AA|)\gg \min(p^{1/3}|A|^{2/3},|A|^{6/5}).
\]
\end{enumerate}
\end{exercise}
The last part of the exercise implies that if $|A|\leq p^{5/8}$, then
\[
\max(|A+A|,|AA|)\gg |A|^{6/5}.
\]
Since $|A|^2/p^{1/2} > p^{1/3}|A|^{2/3}$ when $|A| > p^{5/8}$, the best known sum-product results in $\F_p$ can be summarized as
\begin{equation}
  \label{eq:30}
  \max(|A+A|, |AA|)\gg \min(\sqrt{p|A|}, |A|^2/p^{1/2}, |A|^{6/5}).
\end{equation}

In \cite{stevens2016improved}, Stevens and de Zeeuw use Theorem~\ref{thm:5} in conjunction with a clever induction argument to prove a point-line incidence bound for general point sets $P\subseteq\F_p^2$.
Namely, that for any set of lines $L$ in $\F_p^2$ such that $|P|^{7/8} < |L| < |P|^{8/7}$ and $|L|^{13}\ll p^{15}|P|^2$,
\[
I(P,L)\ll |P|^{11/15}|L|^{11/15}.
\]
Further applications of this bound and Theorem~\ref{thm:5} may be found in \cite{stevens2016improved}.

\subsection{A bound for the number of collinear triples in \(P=A\times A\)}
\label{sec:orgheadline16}

Given a subset $A$ of $\F_p$, let $T(A)$ denote the number of \emph{collinear triples}\/ of points in $P=A\times A$.

For any set $A$, we have $T(A)\ll |A|^5$, which we may see as follows.
Three points $(a,a'),(b,b'), (c,c')$ in $P=A\times A$ are collinear if
 \begin{equation}
   \label{eq:z25}
   \det
\begin{pmatrix}
  1 & 1 & 1\\
  a & b & c \\
  a'&b'&c'\\
\end{pmatrix}
=0.
 \end{equation}
Evaluating the determinant yields the equation
\begin{equation}
  \label{eq:31}
  (b-a)(c'-a')=(b'-a')(c-a).
\end{equation}
Since we have six variables in $|A|^6$ and one equation, we have $\ll |A|^5$ solutions.

Recall that to find lower bounds for $|A+AA|$ and $|A(A+A)|$, we found upper bounds for six variable energy-type equations.
It turns out that \eqref{eq:31} can be bounded in a similar way, leading to the following bound, due to \cite{yazici2015growth}, see also \cite{petridis2016collinear,petridis2016bounds}.
\begin{thm}
  \label{thm:6}
Let $A$ be a subset of $\F_p$.
If \(|A|\ll p^{2/3}\), then
\[
T(A)\ll \frac{|A|^6}p+|A|^{9/2}.
\]
\end{thm}

\begin{proof}
If $a,b\not=c$ and $a',b'\not=c$, then equation~\eqref{eq:31} reduces to
\begin{equation}
  \label{eq:z26}
  \frac{b-a}{c-a}=\frac{b'-a'}{c'-a'}.
\end{equation}
Since the number of collinear triples where $a=c, b=c, a'=c',$ or $b'=c'$ is $O(|A|^4)$, we have
\begin{equation}
  \label{eq:z27}
  T(A) = \left|\left\{(a,\ldots,c')\in A^6\colon   \frac{b-a}{c-a}=\frac{b'-a'}{c'-a'} \neq 0,\infty\right\}\right|+O(|A|^4).
\end{equation}

Thus to bound $T(A)$, it suffices to count the number of solutions to \eqref{eq:z26} with $a,b,c,a',b',c'$ in $A$.
We apply Theorem~\ref{thm:3} to \eqref{eq:z26}.

Let 
\[
Q=\{(1/(c-a), -a/(c-a))\colon a,c\in A\}.
\]

By \eqref{eq:z27} and our definition of $Q$, it follows that $T(A) = E(Q;A)+O(|A|^4)$.
The proposition will follow from Theorem~\ref{thm:3} if we can show that $|Q|=|A|^2$ and $k\leq |A|$, since then
\[
E(Q;A)\ll \frac{|A|^6}p+ (|A|^3)^{3/2}+|A|^4 \ll \frac{|A|^6}p+ |A|^{9/2}.
\]

First $|Q|=|A|^2$, since every $(x,y)\in Q$ corresponds to a unique pair $(c,a)$ in $A\times A$, where
\[
a=-\frac yx \quad\mbox{and}\quad c =\frac 1x -\frac yx.
\]
Second, to show that $k\leq |A|$ we must show that at most $k$ points of $Q$ are collinear.
Consider the linear equation $\alpha x+\beta y=\gamma$ with $\alpha,\beta,$ and $\gamma$ fixed; suppose one of $\alpha,\beta$  equals $1$.
Plugging in $x=1/(c-a)$ and $y=-a/(c-a)$ yields the equation
\[
\alpha-\beta a=\gamma(c-a),
\]
which has at most $|A|$ solutions $(a,c)$, as required.
\end{proof}

The number of collinear triples $T(A)$ can be expressed in terms of the multiplicative energy of shifts of $A$:
\begin{equation}
  \label{eq:33}
  T(A) = \sum_{a,a'\in A}E^\times(A-a,A-a').
\end{equation}
This is easy to see from \eqref{eq:31}.
We first learned of equation~\eqref{eq:33} in \cite{roche-newton2015short}, and the proof there inspired the proof of Theorem~\ref{thm:6}.

The following easy corollary was used in \cite{yazici2015growth} to prove an incidence bound for points and lines (which has since been subsumed by Theorem~\ref{thm:5}).
\begin{cor}
  \label{cor:3}
Let $A$ be a subset of $\F_p$ with $|A|<p^{2/3}$ and let $L_k$ denote the set of lines containing at least $k$ points of $P=A\times A$.
If $k > 3$, then
\[
|L_k| \ll\frac{|A|^{9/2}}{k^3}.
\]
\end{cor}

\begin{proof}
We have
\[
{k\choose 3}|L_k| \leq \sum_{\ell\in L_k}{|P\cap \ell|\choose 3} \ll T(A) \ll |A|^{9/2}.
\]
Since $k > 3$, we have ${k\choose 3}\gg k^3$, so the bound follows.
\end{proof}
Theorem~\ref{thm:5} implies that
\begin{equation}
  \label{eq:32}
  |L_k| \ll \frac{|A|^5}{k^4}
\end{equation}
for $k > |A|^{3/2}/p^{1/2}$.
In Lemma~\ref{lem:2}, we show that the same bound actually holds whenever $k>2|A|^2/p$.
The bound \eqref{eq:32} is essentially equivalent to the statement that for $|A|<p^{2/3}$, the point set $A\times A$ determines $\ll |A|^5\log(|A|)$ collinear \emph{quadruples}.
Given such a bound for collinear quadruples, we may recover \eqref{eq:32} by the same method used to prove Corollary~\ref{cor:3}.
See \cite{petridis2016collinear} for further discussion.

\section{An expander below the \(p^{2/3}\) threshold}
\label{sec:orgheadline20}

In this section, we prove the following theorem due to the second listed author~\cite{petridis2016products}:
\begin{thm}
  \label{thm:7}
Let $p$ be a prime and let $A$ be a subset of $\F_p$.
Then the number of solutions to
\begin{equation}
  \label{eq:9}
  (a-b)(c-d)=(a'-b')(c'-d')\quad\mbox{with $a,b,c,d,a',b',c',d'$ in $A$}
\end{equation}
is $|A|^8/p + O(p^{2/3}|A|^{16/3})$.

Hence if $|A|\gg p^{5/8}$, then the number of solution is $O(|A|^8/p)$ and hence
\[
|(A-A)(A-A)|\gg p.
\]
\end{thm}
This result is that it says that $|(A-A)(A-A)|$ is nearly as large as possible when $|A|$ is at least $p^{5/8}$, which is lower than the $p^{2/3}$ threshold.
Subsequently, Rudnev, Shkredov, and Stevens \cite{rudnev2016energy} proved that
\[
\left| \left\{ \frac{ab-c}{a-d}\colon a,b,c,d\in A\right\}\right| \gg p
\]
whenever $|A|\gg p^{25/42- o(1)}$, which also breaks the $p^{2/3}$ threshold.
Recently, the authors, together with Roche-Newton, Rudnev, and Shkredov \cite{murphy2017results} have proved several results that pass the $p^{2/3}$ threshold.
For instance,
\[
|R[A]| = \left| \left\{ \frac{b-a}{c-a}\colon a,b,c\in A\right\}\right| \gg p
\]
whenever $|A|\geq p^{3/5}$, and
\[
|R[A]| \gg \frac{|A|^{8/5}}{\log^2(|A|)}
\]
whenever $|A|\leq p^{5/12}$.

\begin{proof}[Proof of Theorem~\ref{thm:7}]
As before, we use an energy-type argument: let $r(x)=r_{(A-A)(A-A)}(x)$.
Then $r(x)$ is supported on $(A-A)(A-A)$ and $\sum_x r(x) = |A|^4$, thus
\[
|A|^8 \leq |(A-A)(A-A)|\sum_{x} r^2(x).
\]
The second moment of $r(x)$ counts solutions to equation~\eqref{eq:9}.

There are $O(|A|^6)$ solutions where either side of \eqref{eq:9} is zero, thus we have
\begin{equation}
  \label{eq:10}
\sum_x r^2(x) = \left|\left\{\frac{a-b}{a'-b'}=\frac{c-d}{c'-d'}\not=0,\infty\right\}\right| + O(|A|^6).  
\end{equation}
We can write this quantity as a second moment of a different function, which we will call $Q_\xi$:
\begin{equation}
  \label{eq:11}
  Q_\xi := \left|\left\{(a,b,c,d)\in A^4\colon\frac{a-b}{c-d}=\xi\right\}\right|.
\end{equation}
Then by \eqref{eq:10} and \eqref{eq:11} we have
\begin{equation}
  \label{eq:12}
  \sum_x r^2(x) = \sum_{\xi\not=0}Q_\xi^2 + O(|A|^6).
\end{equation}

The following lemma provides the necessary bound for the second moment of $Q_\xi$:
\begin{lem}
  \label{lem:4}
\[
\sum_{\xi\not=0}Q_\xi^2 \leq \frac{|A|^8}p + O(p^{2/3}|A|^{16/3}).
\]
\end{lem}
We defer the proof of Lemma~\ref{lem:4}, and finish the proof of Theorem~\ref{thm:7}.

Combining \eqref{eq:12} with Lemma~\ref{lem:4} yields
\[
  \sum_x r^2(x) \leq  \frac{|A|^8}p + O(|A|^6 + p^{2/3}|A|^{16/3}).
\]
Since $|A|^6\ll p^{2/3}|A|^{16/3}$ for all $A$, we have
\begin{equation}
  \label{eq:21}
  \sum_x r^2(x) \leq  \frac{|A|^8}p + O(p^{2/3}|A|^{16/3}),
\end{equation}
as claimed.

If $|A| \geq p^{5/8}$, then $\sum_x r^2(x) \ll |A|^8/p$, so $|(A-A)(A-A)| \gg p$.
\end{proof}

Now we prove Lemma~\ref{lem:4}.
\begin{proof}[Proof of Lemma~\ref{lem:4}]
To begin, we record some basic facts about $Q_\xi$ and introduce a related quantity, $E_\xi$.
For $\xi\not=0$, we have
\begin{equation}
  \label{eq:13}
  Q_\xi = |\{(a,b,c,d)\in A^4\colon a-\xi c = b-\xi d, \,a\not=b, c\not=d\}| = E^+(A,\xi A) - |A|^2.
\end{equation}
Since
\[
\sum_{\xi\not=0}Q_\xi = |A|^2(|A|-1)^2,
\]
we have
\begin{equation}
  \label{eq:14}
  \sum_{\xi\in X}E^+(A,\xi A) =  \sum_{\xi\in X} \left( Q_\xi + |A|^2 \right) \leq |A|^4 + |X||A|^2.
\end{equation}

It follows from \eqref{eq:14} that if we set
\[
E_\xi = E^+(A,\xi A) -\frac{|A|^4}p,
\]
then
\begin{equation}
  \label{eq:15}
\sum_{\xi\not=0}E_\xi \leq p|A|^2.  
\end{equation}
The quantity $E_\xi$ is useful because it is \emph{non-negative}: by Cauchy-Schwarz,
\[
E^+(A,\xi A)\geq\frac{|A|^4}{|A\pm \xi A|}\geq \frac{|A|^4}p.
\]

Now we will estimate the second moment of $Q_\xi$.
To begin, we replace one power of $Q_\xi$ by $E_\xi$ and estimate the error:
\begin{align*}
  \sum_{\xi\not=0}Q_\xi^2&= \sum_{\xi\not=0} Q_\xi \left( E^+(A,\xi A) - |A|^2 \right)\\
&= \sum_{\xi\not=0} Q_\xi \left( E_\xi +\frac{|A|^4}p - |A|^2 \right)\\
&\leq \sum_{\xi\not=0} Q_\xi E_\xi +\frac{|A|^4}p\sum_{\xi\not=0} Q_\xi\\
&\leq \frac{|A|^8}p +\sum_{\xi\not=0} Q_\xi E_\xi.
\end{align*}
Thus by \eqref{eq:12},
\begin{equation}
  \label{eq:20}
\sum_x r^2(x) =  \sum_{\xi\not=0}Q_\xi^2 +O(|A|^6) \leq  \frac{|A|^8}p +\sum_{\xi\not=0} Q_\xi E_\xi + O(|A|^6).
\end{equation}

Now, to estimate the sum over $\xi$, we divide into two cases.
Let $B_K=\{\xi\not=0\colon Q_\xi > |A|^3/K\}$.
Then
\begin{equation}
  \label{eq:16}
  \sum_{\xi\not=0}Q_\xi E_\xi \leq \sum_{\xi\in B_K}Q_\xi E_\xi + \frac{|A|^3}K \sum_{\xi\not=0} E_\xi= I + II.
\end{equation}
We bound second term by \eqref{eq:15}:
\begin{equation}
  \label{eq:17}
  II = \frac{|A|^3}K \sum_{\xi\not=0} E_\xi \leq \frac{p|A|^5}K.
\end{equation}

To bound the first term, we use the trivial bound $|Q_\xi|\leq |A|^3$ to find
\begin{equation}
  \label{eq:18}
  I =\sum_{\xi\in B_K}Q_\xi E_\xi  \leq |A|^3 \sum_{\xi\in B_K}E_\xi \leq |A|^3\sum_{\xi\in B_K}E^+(A,\xi A).
\end{equation}
To bound this last sum, we use the following Lemma, which we will prove in the next section.
\begin{lem}
  \label{lem:3}
If $|A|\ll p^{2/3}$, then for any $X\subseteq \F_p$ such that $|X|\leq |A|^3$,
\[
\sum_{\xi\in X}E^+(A,\xi A) \ll |A|^3|X|^{2/3}.
\]
\end{lem}
Since
\[
\frac{|A|^3}K |B_K| < \sum_{\xi\in B_K}Q_\xi\leq |A|^4
\]
and $K \leq |A|$, we have
\[
|B_K|\leq |A|^2.
\]
Thus we may apply Lemma~\ref{lem:3} with $X=B_K$.

By Lemma~\ref{lem:3} and \eqref{eq:18},
\begin{equation}
  \label{eq:22}
  I \ll |A|^6|B_K|^{2/3}.
\end{equation}
Now we use Lemma~\ref{lem:3} again to bound $|B_K|$:
\[
\frac{|A|^3}K |B_K| \leq \sum_{\xi\in B_K}E^+(A,\xi A) \ll |A|^{3}|B_K|^{2/3},
\]
hence $|B_K|\ll K^3$.

Combining the bounds for $I$ and $II$ with the bound $|B_K|\ll K^3$, we have
\[
  \sum_{\xi\not=0}Q_\xi E_\xi \ll K^2|A|^6 + \frac{p|A|^5}K.
\]
To balance the terms on the right-hand side of the previous equation, we set $K=(p/|A|)^{1/3}$:
\begin{equation}
  \label{eq:19}
  \sum_{\xi\not=0}Q_\xi E_\xi \ll p^{2/3}|A|^{16/3}.
\end{equation}
This completes the proof of Lemma~\ref{lem:4}, pending the proof of Lemma~\ref{lem:3}.
\end{proof}

\subsection*{Proof of Lemma~\ref{lem:3}}
Recall that Lemma~\ref{lem:3} states that if $|A|\ll p^{2/3}$, then for any set $X\subseteq\F_p$ such that $|X|\leq |A|^3$, we have 
\[
\sum_{\xi\in X}E^+(A,\xi A) \ll |A|^3|X|^{2/3}.
\]
This is an explicit version of Bourgain's Theorem C from \cite{bourgain2009multilinear}.
Similar results were proved over $\R$ in \cite{murphy2015variations} by the \ST{} incidence bound.
We use the same approach as \cite{murphy2015variations}, but we use the following lemma in place of the \ST{} theorem.
\begin{lem}
  \label{lem:2}
Let $A$ be a subset of $\F_p$ and let $L_t$ denote the set of lines in $\F_P^2$ that contain at least $t$ points of $P=A\times A$.
If $t > \min(2|A|^2/p,1)$, then
\[
|L_t| \ll \frac{|A|^5}{t^4}.
\]
\end{lem}
The proof of Lemma~\ref{lem:2} requires the following bound, which is implicit in the work of Bourgain, Katz, and Tao~\cite{bourgain2004sum-product} and appears explicitly in \cite{murphy2016point-line}:
\begin{equation}
  \label{eq:34}
\sum_{\text{all lines }\ell} \left( i(\ell)-\frac{|A|^2}p \right)^2 \leq p|A|^2,
\end{equation}
where $i(\ell)=|(A\times A)\cap\ell|$.

\begin{proof}[Proof of Lemma~\ref{lem:2}]
For a line $\ell$ in $\F_p^2$, let $i(\ell)=|P\cap \ell|$, where $P=A\times A$.
Thus if $\ell\in L_t$, then $i(\ell)\geq t$.

Since $t>2|A|^2/p$, we have
\[
i(\ell)-\frac{|A|^2}p \geq \frac t2
\]
for all $\ell$ in $L_t$.
Thus
\[
\frac{|L_t|t^2}4 \leq \sum_{\ell\in L_t} \left( i(\ell)-\frac{|A|^2}p \right)^2.
\]
On the other hand, by equation~\eqref{eq:34} the right-hand side of the previous equation is at most $p|A|^2$, so
\[
|L_t| \ll \frac{p|A|^2}{t^2}.
\]

Now we consider two cases.
If $t \leq c|A|^{3/2}/p^{1/2}$, we have
\[
|L_t| \leq \frac{c^2|A|^3}{pt^2}|L_t| \ll \frac{|A|^5}{t^4}.
\]
If $t \geq c|A|^{3/2}/p^{1/2}$, then we will apply Theorem~\ref{thm:5}.
Since
\[
t|L_t| \leq I(P,L_t),
\]
by Theorem~\ref{thm:5}, we have
\[
t|L_t| \ll \frac{|A|^{3/2}|L_t|}{p^{1/2}} + |A|^{5/4}|L_t|^{3/4} + |A|^2.
\]
Since $t \geq c|A|^{3/2}/p^{1/2}$, if $c$ is sufficiently large (depending on the implicit constants in Theorem~\ref{thm:5}), we have
\[
t|L_t| \ll  |A|^{5/4}|L_t|^{3/4} + |A|^2,
\]
hence
\[
|L_t| \ll \frac{|A|^5}{t^4} + \frac{|A|^2}t \ll \frac{|A|^5}{t^4}.
\]
(The last inequality follows because $t\leq |A|$.)

Finally, note that if $1< t\ll 1$, then $|L_t|\ll |A|^5/t^4$ is trivial, since $|L_t|\leq |A|^4$.
\end{proof}

Now we proceed to the proof of the main result of this section.
\begin{proof}[Proof of Lemma~\ref{lem:3}]
To show that
\[
S:=\sum_{\xi\in X}E^+(A,\xi A)\ll |A|^3|X|^{2/3},
\]
we first write
\[
S = \sum_{\xi\in X}\sum_y r_{A+\xi A}^2(y).
\]

Let $Z_j$ denote the set of pairs $\{(\xi,y)\colon r_{A+\xi A}(y) >\Delta 2^j\}$.
Then
\begin{equation}
  \label{eq:26}
  S \ll \Delta |X||A|^2 + \sum_{j\geq 0}|Z_j|(\Delta 2^j)^2.
\end{equation}
On the other hand, for each pair $(\xi,y)$ in $Z_j$, we may associate the line $\ell_{\xi,y}=\{(a,b)\colon a+\xi b = y\}$.
Since the line $\ell_{\xi,y}$ contains at least $\Delta 2^j$ points of $A\times A$, by Lemma~\ref{lem:3} we have
\begin{equation}
  \label{eq:28}
  |Z_j| \leq |L_j| \ll \frac{|A|^5}{(\Delta 2^j)^4},
\end{equation}
whenever $\Delta 2^j \geq\min( 2|A|^2/p,1)$.
(We do not need strict inequality because it is included in the definition of $Z_j$.)

Assume for now that $\Delta \geq \min(2|A|^2/p,1)$; at the end of the argument, we will prove that our choice of $\Delta$ satisfies this condition whenever $|A|\ll p^{2/3}$.
By \eqref{eq:26} and \eqref{eq:28}, we have
\[
S \ll \Delta |X||A|^2 +\sum_{j\geq 0}(\Delta 2^j)^2\frac{|A|^5}{(\Delta 2^j)^4},
\]
Thus
\[
S \ll \Delta |X||A|^2 + \frac{|A|^5}{\Delta^2}.
\]
Choosing $\Delta = |A|/|X|^{1/3}$ yields
\[
S \ll |A|^3|X|^{2/3},
\]
as desired.

Now we will check that $\Delta = |A|/|X|^{1/3}$ is at least $2|A|^2/p$ whenever $|A|\ll p^{2/3}$:
\[
\Delta = \frac{|A|}{|X|^{1/3}} \geq \frac{2|A|^2}p \iff |X| \ll \frac{p^3}{|A|^3}.
\]
On the other hand, if $|A|\ll p^{2/3}$, then $p^3/|A|^3\gg p \geq |X|$.
Finally, $|X|\leq |A|^3$ implies $\Delta \geq 1$.
\end{proof}


\begin{thebibliography}{10}

\bibitem{bennett2013group}
M.~Bennett, D.~Hart, A.~Iosevich, J.~Pakianathan, and M.~Rudnev.
\newblock Group actions and geometric combinatorics in {$\mathbb{F}_q^d$}.
\newblock 11 2013.

\bibitem{bourgain2009variant}
J.~Bourgain and M.~Z. Garaev.
\newblock On a variant of sum-product estimates and explicit exponential sum
  bounds in prime fields.
\newblock {\em Math. Proc. Cambridge Philos. Soc.}, 146(1):1--21, 2009.

\bibitem{bourgain2004sum-product-a}
J.~Bourgain, N.~Katz, and T.~Tao.
\newblock A sum-product estimate in finite fields, and applications.
\newblock {\em Geometric and Functional Analysis}, 14(1):27--57, 2004.

\bibitem{bourgain2004sum-product}
J.~Bourgain, N.~Katz, and T.~Tao.
\newblock A sum-product estimate in finite fields, and applications.
\newblock {\em Geom. Funct. Anal.}, 14(1):27--57, 2004.

\bibitem{bourgain2009multilinear}
Jean Bourgain.
\newblock Multilinear exponential sums in prime fields under optimal entropy
  condition on the sources.
\newblock {\em Geometric and Functional Analysis}, 18(5):1477--1502, 2009.

\bibitem{chapman2009pinned}
Jeremy Chapman, M.~Burak Erdogan, Derrick Hart, Alex Iosevich, and Doowon Koh.
\newblock {P}inned distance sets, k-simplices, {W}olff's exponent in finite
  fields and sum-product estimates, 2009.

\bibitem{zeeuw2016short}
Frank de~Zeeuw.
\newblock {A} short proof of {R}udnev's point-plane incidence bound.
\newblock {\em arXiv preprint arXiv:1612.02719}, 2016.

\bibitem{elekes1997number}
Gy{{\"o}}rgy Elekes.
\newblock On the number of sums and products.
\newblock {\em Acta Arith.}, 81(4):365--367, 1997.

\bibitem{erdos1983products-a}
P.~Erdos and E.~Szemer{\'e}di.
\newblock On sums and products of integers.
\newblock {\em Studies in pure mathematics}, pages 213--218, 1983.

\bibitem{garaev2007explicit}
M.~Z. Garaev.
\newblock An explicit sum-product estimate in {$\Bbb F_p$}.
\newblock {\em Int. Math. Res. Not. IMRN}, (11):Art. ID rnm035, 11, 2007.

\bibitem{garaev2008sum-product}
M.~Z. Garaev.
\newblock The sum-product estimate for large subsets of prime fields.
\newblock {\em Proc. Amer. Math. Soc.}, 136(8):2735--2739, 2008.

\bibitem{glibichuk2007additive}
AA~Glibichuk and SV~Konyagin.
\newblock Additive properties of product sets in fields of prime order.
\newblock {\em Additive Combinatorics, CRM Proceedings and Lecture Notes},
  43:279--286, 2007.

\bibitem{hart2011averages}
Derrick Hart, Alex Iosevich, Doowon Koh, and Misha Rudnev.
\newblock Averages over hyperplanes, sum-product theory in vector spaces over
  finite fields and the {E}rd{\H o}s-{F}alconer distance conjecture.
\newblock {\em Trans. Amer. Math. Soc.}, 363(6):3255--3275, 2011.

\bibitem{hart2007sum-product}
Derrick Hart, Alex Iosevich, and Jozsef Solymosi.
\newblock Sum-product estimates in finite fields via {K}loosterman sums.
\newblock {\em Int. Math. Res. Not. IMRN}, (5):Art. ID rnm007, 14, 2007.

\bibitem{katz2008slight}
Nets~Hawk Katz and Chun-Yen Shen.
\newblock A slight improvement to {G}araev's sum product estimate.
\newblock {\em Proc. Amer. Math. Soc.}, 136(7):2499--2504, 2008.

\bibitem{lund2014incidence}
Ben Lund and Shubhangi Saraf.
\newblock {I}ncidence {B}ounds for {B}lock {D}esigns, 2014.

\bibitem{murphy2016point-line}
Brendan Murphy and Giorgis Petridis.
\newblock {A} point-line incidence identity in finite fields, and applications.
\newblock {\em Mosc. J. Comb. Number Theory}, 6(1):64--95, 2016.

\bibitem{murphy2017results}
Brendan Murphy, Giorgis Petridis, Oliver Roche-Newton, Misha Rudnev, and
  Ilya~D. Shkredov.
\newblock New results on sum-product type growth in positive characteristic.
\newblock {\em preprint}, 2017.

\bibitem{murphy2015variations}
Brendan Murphy, Oliver Roche-Newton, and Ilya Shkredov.
\newblock Variations on the sum-product problem.
\newblock {\em SIAM J. Discrete Math.}, 29(1):514--540, 2015.

\bibitem{petridis2016collinear}
Giorgis Petridis.
\newblock {C}ollinear triples and quadruples for {C}artesian products in
  $\mathbb{{F}}_p^2$, 2016.

\bibitem{petridis2016pinned}
Giorgis Petridis.
\newblock {P}inned algebraic distances determined by {C}artesian products in
  $\mathbb{{F}}_p^2$, 2016.

\bibitem{petridis2016products}
Giorgis Petridis.
\newblock {P}roducts of {D}ifferences in {P}rime {O}rder {F}inite {F}ields,
  2016.

\bibitem{petridis2016bounds}
Giorgis Petridis and Igor~E. Shparlinski.
\newblock {B}ounds on trilinear and quadrilinear exponential sums, 2016.

\bibitem{pham2016three-variable}
Thang Pham, Le~Anh Vinh, and Frank de~Zeeuw.
\newblock {T}hree-variable expanding polynomials and higher-dimensional
  distinct distances.
\newblock {\em arXiv preprint arXiv:1612.09032}, 2016.

\bibitem{roche-newton2015short}
Oliver Roche-Newton.
\newblock {A} short proof of a near-optimal cardinality estimate for the
  product of a sum set, 2015.

\bibitem{roche-newton2016sum-product}
Oliver Roche-Newton, Misha Rudnev, and Ilya~D. Shkredov.
\newblock New sum-product type estimates over finite fields.
\newblock {\em Adv. Math.}, 293:589--605, 2016.

\bibitem{rudnev2012improved}
Misha Rudnev.
\newblock An improved sum-product inequality in fields of prime order.
\newblock {\em Int. Math. Res. Not. IMRN}, (16):3693--3705, 2012.

\bibitem{rudnev2014number}
Misha Rudnev.
\newblock {O}n the number of incidences between planes and points in three
  dimensions.
\newblock To appear in Combinatorica, 2014.

\bibitem{rudnev2016energy}
Misha Rudnev, Ilya~D. Shkredov, and Sophie Stevens.
\newblock {O}n the energy variant of the sum-product conjecture, 2016.

\bibitem{shkredov2016question}
Ilya~D. Shkredov.
\newblock On a question of {A}. {B}alog.
\newblock {\em Pacific J. Math.}, 280(1):227--240, 2016.

\bibitem{stevens2016improved}
Sophie Stevens and Frank de~Zeeuw.
\newblock An improved point-line incidence bound over arbitrary fields.
\newblock {\em arXiv preprint}, 2016.

\bibitem{vinh2011szemeredi-trotter}
Le~Anh Vinh.
\newblock The {S}zemer\'edi-{T}rotter type theorem and the sum-product estimate
  in finite fields.
\newblock {\em European J. Combin.}, 32(8):1177--1181, 2011.

\bibitem{yazici2015growth}
Esen~Aksoy Yazici, Brendan Murphy, Misha Rudnev, and Ilya Shkredov.
\newblock {G}rowth {E}stimates in {P}ositive {C}haracteristic via {C}ollisions.
\newblock {\em Int. Math. Res. Not. IMRN}, 2016.

\end{thebibliography}

\end{document}